\documentclass[sn-mathphys-num]{sn-jnl}

\usepackage{enumerate,xfrac,faktor,lineno,nicefrac,float,amssymb}
\usepackage{booktabs,multirow}

\usepackage{xcolor}

\usepackage{graphicx}%
\usepackage{multirow}%
\usepackage{amsmath,amssymb,amsfonts}%
\usepackage{amsthm}%
\usepackage{mathrsfs}%
\usepackage[title]{appendix}%
\usepackage{xcolor}%
\usepackage{textcomp}%
\usepackage{manyfoot}%
\usepackage{booktabs}%
\usepackage{algorithm}%
\usepackage{algorithmicx}%
\usepackage{algpseudocode}%
\usepackage{listings}%

\theoremstyle{thmstyleone}%
\newtheorem{theorem}{Theorem}
\newtheorem{proposition}[theorem]{Proposition}%
\newtheorem{corollary}[theorem]{Corollary}

\theoremstyle{thmstyletwo}%
\newtheorem{example}{Example}%

\theoremstyle{thmstylethree}%
\newtheorem{definition}{Definition}%

\begin{document}

\title[]{On the Cosine Similarity and Orthogonality between Persistence Diagrams}


\author*{\fnm{Azmeer} \sur{Nordin}}\email{nfazmeer.nordin@ukm.edu.my}

\author{\fnm{Mohd Salmi Md} \sur{Noorani}}\email{msn@ukm.edu.my}
\equalcont{These authors contributed equally to this work.}

\author{\fnm{Nurulkamal} \sur{Masseran}}\email{kamalmsn@ukm.edu.my}
\equalcont{These authors contributed equally to this work.}

\author{\fnm{Mohd Sabri} \sur{Ismail}}\email{sabriismail@ukm.edu.my}
\equalcont{These authors contributed equally to this work.}

\author{\fnm{Nur Firyal} \sur{Roslan}}\email{nurfiryal@ukm.edu.my}
\equalcont{These authors contributed equally to this work.}

\affil{\orgdiv{Department of Mathematical Sciences}, \orgname{Universiti Kebangsaan Malaysia}, \orgaddress{\city{Bangi}, \postcode{43600}, \state{Selangor}, \country{Malaysia}}}


\abstract{
	Topological data analysis is an approach to study shape of a data set by means of topology. Its main object of study is the persistence diagram, which represents the topological features of the data set at different spatial resolutions. Multiple data sets can be compared by the similarity of their diagrams to understand their behaviors in relative to each other. The bottleneck and Wasserstein distances are often used as a tool to indicate the similarity. In this paper, we introduce cosine similarity as a new indicator for the similarity between persistence diagrams and investigate its properties. Furthermore, it leads to the new notion of orthogonality between persistence diagrams. It turns out that the orthogonality refers to perfect dissimilarity between persistence diagrams under the cosine similarity. Through data demonstration, the cosine similarity is shown to be more accurate than the standard distances to measure the similarity between persistence diagrams.
}



\maketitle

\section{Introduction}\label{section: introduction}
Topological data analysis is a fast-growing field for analyzing complex data via theories and techniques from topology. It extracts information from the topological structure or shape of the data which may be obscured by traditional methods in data analysis. It is adaptable to various types of data and also robust to noise \cite{Otter}. Due to its advantages, it has seen wide applications in diverse fields \cite{Wasserman,Chazal_1}, such as enhancing classical machine learning models \cite{Hensel}, identifying cancer tumors in oncology \cite{Skaf}, detecting warning signs of financial crashes in stock markets \cite{Gidea}, and recognizing patterns in multimedia data \cite{Chunyuan,Reise}. These are just a few examples, but the list is endless.

The central tool in topological data analysis is the persistent homology \cite{Edelsbrunner}, which is the application of homology theory on a data set. In particular, it constructs shape of the data set via a certain mathematical structure called a simplicial complex, and then computes its topological features such as connected components, holes and cavities at different spatial resolutions. The existences of those features along the resolutions are recorded in a scatter plot, which is referred as a persistence diagram. Further analysis is conducted on the diagram to obtain meaningful information depending on purposes of the study.

Many studies pertaining to applications of persistent homology will involve working with multiple data sets. Their persistence diagrams can be compared with each other in a certain way to examine whether the data sets have similar structures or behaviors. For instance, this is useful to distinguish an abnormal data set under an extreme event in a particular situation. This is demonstrated in some studies of air pollution based on daily data of Air Quality Index, in which any month with haze episode is characterized by having persistence diagram with abundant features associated to holes \cite{Fariha_1,Fariha_2}. The comparison between persistence diagrams in these works are done simply by observing their numbers of features without the need of other machinery.

Although there are more, the above example is enough to emphasize the importance of comparing persistence diagrams to understand the behaviors of data sets in question. To be more precise, we compare diagrams in terms of their similarity i.e. how close the points of one diagram to those of another one on the same axes. A pair of diagrams is considered similar if both collections of points are sufficiently close, while dissimilar if both are too far apart. This depends on the tool or indicator used to quantify similarity and ascertain whether two diagrams are similar or dissimilar. 

The common way to measure the similarity is to use the bottleneck and Wasserstein distances for persistence diagrams \cite{Edelsbrunner,Dey}. Specifically, two diagrams are deemed similar if either distance is small. This idea of similarity via distances has been utilized in \cite{Bando} to detect causality between two variables via bottleneck distance, and in \cite{Fariha_3} to classify clusters between several data sets via Wasserstein distance for topological clustering.

However in certain cases, the distances may not be accurate to signify the similarity. It is possible for two persistence diagrams which are closely similar to have a noticeably large distance, while those that are clearly dissimilar may record a very small distance (see Examples \ref{example: distance not accurate 1} and \ref{example: distance not accurate 2}). For this reason, a few alternative tools or indicators have been designed in the literature to measure the similarity between persistence diagrams more accurately, such as the RST parametric model \cite{Agami} and the vectorization of persistence diagrams into complex polynomials \cite{DiFabio}.

In this paper, we propose a new indicator called cosine similarity as an alternative to measure the similarity between persistence diagrams. This is motivated by the same notion in Euclidean space, in which the cosine similarity for non-zero vectors $\vec{u},\vec{v} \in \mathbb{R}^d$ is defined as
\begin{equation*}
	S_C(\vec{u},\vec{v})=\cos\alpha = \frac{\vec{u} \cdot \vec{v}}{\|\vec{u}\|\|\vec{v}\|}
\end{equation*}
where $\alpha$ denotes the angle between them. It indicates how close the vectors are aligned in the same direction regardless of their magnitudes. It is widely applied in text mining to measure the similarity between two documents in terms of subject matter \cite{Wang}, and in data mining to measure the cohesion within data clusters \cite{Tan}.

In our approach, we employ persistence landscapes \cite{Bubenik_1,Bubenik_2}, which are an alternative presentation of persistence diagrams, to use the above formula for our cosine similarity. Each landscape is a sequence of piecewise linear functions and hence can be associated to a certain norm and inner product, unlike a persistence diagram. The norm for these landscapes is commonly utilized as a moving summary in persistent homology analysis of time series in various situations such as river water levels \cite{Sadiq_1,Sadiq_2} and financial data \cite{Sabri_1,Sabri_2}, but the applications of its inner product are seldom explored in the literature.

As a related result here, the inner product motivates a new notion of orthogonality between persistence diagrams. In particular, we say that two diagrams are orthogonal if their landscapes admit a zero inner product. It turns out that the orthogonality refers to perfect dissimilarity between two persistence diagrams under the cosine similarity i.e. both diagrams are considered entirely different from each other.

To sum up, the aim of this paper is to introduce the cosine similarity for persistence diagrams and investigate its properties, which are discussed in Section \ref{section: cosine similarity}. Moreover, we explore the notion of orthogonality for the diagrams and provide some justifications on why it means perfect dissimilarity under the cosine similarity in Section \ref{section: orthogonality}. Finally, we compare the accuracy of cosine similarity against the bottleneck, Wasserstein and other relevant distances to measure similarity of persistence diagrams through a data demonstration in Section \ref{section: data demo}.

\section{Persistent Homology}\label{section: persistent homology}
This section provides a brief introduction to the persistent homology which covers some common notions such as simplicial complexes, homology groups, filtrations, persistence diagrams and persistence landscapes. A detailed account of this theory can be referred to \cite{Edelsbrunner,Dey}.

Let $P$ be a non-empty finite set. It shall be called a \emph{point set}. An \emph{abstract simplicial complex} $\mathcal{K}$ is a collection of non-empty subsets of $P$ such that for any set $\sigma \in \mathcal{K}$, its non-empty subset $\tau \subseteq \sigma$ is also in $\mathcal{K}$. Each set $\sigma \in \mathcal{K}$ of size $k+1$ is called a \emph{$k$-simplex}. Its non-empty subset $\tau \subseteq \sigma$ of size $\ell+1$ is called an \emph{$\ell$-face} of $\sigma$. A $(k-1)$-face of $\sigma$ is called a \emph{facet}.

A complex $\mathcal{K}$ can be geometrically visualized in $\mathbb{R}^{|P|-1}$ by associating an element of $P$ to zero vector while others to standard unit vectors. A simplex $\sigma \in $$\mathcal{K}$ corresponds to the convex hull of the relevant vectors in $\mathbb{R}^{|P|-1}$. Hence, a simplex can be thought as a geometrical shape under this visualization. For example, a $0$-simplex, a $1$-simplex, a $2$-simplex and a $3$-simplex are a point, an edge, a solid triangle and a solid tetrahedron, respectively. This visualization is called a \emph{geometric realization} of $\mathcal{K}$. This is useful if we want to relate certain topological features to a simplicial complex, such as connected components and holes. We shall mention that there is another notion of a complex for a point set in an Euclidean space, which is called a geometric simplicial complex, but it is not required in this paper.

A function $f:\mathcal{K} \rightarrow \mathbb{R}$ is said to be \emph{simplex-wise monotone} if $f(\tau) \leq f(\sigma)$ whenever $\tau$ is a face of $\sigma$. In this case, the subcollection $\mathcal{K}_\theta = \left\{\sigma \in \mathcal{K} \mid f(\sigma) \leq \theta\right\}$ is a subcomplex of $\mathcal{K}$ for any $\theta \in \mathbb{R}$. It is clear that $\mathcal{K}_\theta \subseteq \mathcal{K}_{\theta'}$ for $\theta \leq \theta'$. So, the collection of complexes $\left\{\mathcal{K}_\theta \mid \theta \geq \theta_0\right\}$, where $\theta_0 = \min f$, forms a nested inclusion as $\theta$ increases. It is called the \emph{simplicial filtration} of $\mathcal{K}$ by $f$.

For example, Vietoris-Rips filtration is often considered in topological data analysis due to its simple construction. In this case, the complex $\mathcal{K}$ is taken as the power set of $P$. Suppose that $P$ is equipped with a distance $\mathsf{d}$. We define a simplex-wise monotone function $f$ on $\mathcal{K}$ as
\begin{equation*}
	f(\sigma)=\max _{p, q \in \sigma} \mathsf{d}(p, q).
\end{equation*}
The subcomplex $\mathcal{K}_\theta$ is called the \emph{Vietoris-Rips complex} of parameter $\theta$. In particular, a subset $\sigma \subseteq P$ is in $\mathcal{K}_\theta$ if and only if $\mathrm{d}(p, q) \leq \theta$ for any pair $p, q \in \sigma$. The \emph{Vietoris-Rips filtration} is the collection $\left\{\mathcal{K}_\theta \mid \theta \geq 0\right\}$ in this case.

Given a filtration, we are interested to observe the changes in topological structure of $\mathcal{K}_\theta$ as $\theta$ increases. For this purpose, we shall introduce the homology groups induced by a general complex $\mathcal{K}$. Let $\mathcal{S}_k=\left\{\sigma_1, \sigma_2, \ldots, \sigma_{n_k}\right\}$ be a collection of $k$-simplices in $\mathcal{K}$. We define a formal series of $\mathcal{S}_k$ over the field $\mathbb{Z}_2$ as $\sum_{i=1}^{n_k} a_i \sigma_i$ where $a_i \in \mathbb{Z}_2$. The series is called a \emph{$k$-chain}. The collection of all $k$-chains of $\mathcal{K}$ forms an abelian group, which is called the \emph{$k$\textsuperscript{th} chain group} $\mathsf{C}_k(\mathcal{K})$. We set $\mathsf{C}_k(\mathcal{K})$ as the trivial group if $\mathcal{S}_k$ is empty.

For a $k$-simplex $\sigma$, consider the collection of its facets $\left\{\tau_1, \tau_2, \ldots, \tau_{k+1}\right\}$. We denote $\partial_k(\sigma)=\sum_{i=1}^{k+1} \tau_i$ which is a $(k-1)$-chain of $\mathcal{K}$. We extend this map as a homomorphism $\partial_k: \mathsf{C}_k(\mathcal{K}) \rightarrow \mathsf{C}_{k-1}(\mathcal{K})$ by defining
\begin{equation*}
	\partial_k\left(\sum_{i=1}^{n_k} a_i \sigma_i\right)=\sum_{i=1}^{n_k} a_i \partial_k\left(\sigma_i\right).
\end{equation*}
The kernel of $\partial_k$ is called the \emph{$k$\textsuperscript{th} cycle group} $\mathsf{Z}_k(\mathcal{K})$, while the image of $\partial_{k+1}$ is called the \emph{$k$\textsuperscript{th} boundary group} $\mathsf{B}_k(\mathcal{K})$. The \emph{$k$\textsuperscript{th} homology group} is then defined as the quotient group $\mathsf{H}_k(\mathcal{K})=\mathsf{Z}_k(\mathcal{K}) / \mathsf{B}_k(\mathcal{K})$.

Each homology group relates to a certain type of topological features in a complex $\mathcal{K}$. For example, the zeroth and first homology groups reflect the connected components and holes respectively in $\mathcal{K}$, or to be precise, in its geometric realization \cite{Aktas-Akbas-Fatmaoui}.

Now, consider a simplicial filtration $\left\{\mathcal{K}_\theta \mid \theta \geq \theta_0\right\}$. Since $\mathcal{K}_\theta \subseteq \mathcal{K}_{\theta'}$ for $\theta \leq \theta'$, there is a natural homomorphism $\phi_k^{\theta, \theta'}: \mathsf{H}_k(\mathcal{K}_\theta) \rightarrow \mathsf{H}_k(\mathcal{K}_{\theta'})$ given by
\begin{equation*}
	\phi_k^{\theta, \theta'}: c+\mathsf{B}_k(\mathcal{K}_\theta) \mapsto c+\mathsf{B}_k(\mathcal{K}_{\theta'})
\end{equation*}
for any chain $c \in \mathsf{Z}_k(\mathcal{K}_\theta)$. A non-trivial class in $\mathsf{H}_k(\mathcal{K}_\theta)$ may be reduced to be trivial in $\mathsf{H}_k(\mathcal{K}_{\theta'})$ under this map. This motivates the idea of \emph{birth} and \emph{death} of a non-trivial class in the filtration. It describes a topological feature in $\mathcal{K}_\theta$ as mentioned previously. For this reason, we shall refer a non-trivial class in the filtration simply as a \emph{feature}.

A feature $\xi$ is born at $b \in \mathbb{R}$ if $\xi \in \mathsf{H}_k(\mathcal{K}_b)$ and $\big(\phi_k^{\theta, b}\big)^{-1}(\xi)=\emptyset$ for any $\theta<b$. Next, it dies at $d \in \mathbb{R}$ if $\phi_k^{b, d}(\xi)=\boldsymbol{0}$ but $\phi_k^{b, \theta}(\xi) \neq \boldsymbol{0}$ for any $\theta<d$. It is possible that $\xi$ never dies i.e. $\phi_k^{b, \theta}(\xi) \neq \boldsymbol{0}$ for any $\theta>b$. We denote $d=\infty$ in this case. The \emph{lifespan} of $\xi$ is defined as $\ell=d-b$.

From now on, we focus on the features with respect to $k$\textsuperscript{th} homology group. Let $\xi$ be a feature that is born at $b$ and dies at $d$. It can be represented as a point $(b, d)$. The \emph{persistence diagram} $D$ is the set of points (allowing multiplicities) for those features. It can be visualized as a scatter plot of deaths against births of the features. Note that the points in that plot must lie above the diagonal line $\Delta=\left\{(\theta, \theta) \mid \theta \geq \theta_0\right\}$.

A feature whose point is relatively far from the diagonal line is said to be \emph{persistent}. Such feature has a long lifespan. If the point set is sampled from a manifold, then this feature, such as a connected component or a hole, is expected to exist in that manifold. On the other hand, an impersistent feature in the diagram is simply considered a noise due to the filtration, and the abundance of such features may indicate that the sample itself is affected by noise \cite{Ravishanker}.

We shall introduce two types of distances for persistence diagrams $D_1$ and $D_2$. For this purpose, we shall construct a bijection between the diagrams in a certain way. We allow them to add a finite number of points from $\Delta$ (allowing multiplicities) to create new sets $D_1^{\Delta}$ and $D_2^{\Delta}$ and construct a bijection $\gamma: D_1^{\Delta} \rightarrow D_2^{\Delta}$. For example, we define $\gamma$ as follows:
\begin{equation*}
	\gamma:\left(b_1, d_1\right) \mapsto\left(\frac{b_1+d_1}{2}, \frac{b_1+d_1}{2}\right) \quad \textrm{and} \quad \gamma:\left(\frac{b_2+d_2}{2}, \frac{b_2+d_2}{2}\right) \mapsto\left(b_2, d_2\right)
\end{equation*}
for all points $(b_1, d_1) \in D_1$ and $(b_2, d_2) \in D_2$. We call this bijection as \emph{trivial matching}. It pairs each point in $D_1$ and $D_2$ to the closest point in $\Delta$.

Denote the set of those bijections as $\Gamma$. We allow the expression $\infty-\infty=0$ here. The \emph{bottleneck distance} are defined as
\begin{equation*}
	W_{\infty}(D_1, D_2)=\inf _{\gamma \in \Gamma}\left(\sup _{x \in D_1^{\Delta}}\|x-\gamma(x)\|_{\infty}\right)
\end{equation*}
where $\lVert\cdot\rVert_{\infty}$ denotes supremum norm in $\mathbb{R}^2$. Another alternative is the \emph{$p$-Wasserstein distance} for $p \geq 1$, which is given by
\begin{equation*}
	W_p(D_1, D_2)=\inf _{\gamma \in \Gamma}\left(\sum_{x \in D_1^{\Delta}}{\|x-\gamma(x)\|_\infty}^p\right)^{1 / p}.
\end{equation*}
The bijection $\gamma$ that attains the infimum for either distance is called a \emph{perfect matching}.

The set of possible persistence diagrams do not form a vector space as there is no natural operation that can be defined on the set. Several attempts have been made in the literature to transform those diagrams into suitable vectors in a Banach or Hilbert space. This is called a \emph{vectorization} of persistence diagrams. One such transformation is persistence landscape \cite{Bubenik_1, Bubenik_2}.

For this purpose, we assume that each feature has a finite death. For example, this assumption holds in Vietoris-Rips filtration, except for the zeroth homology group. In this case, there is exactly one feature that never dies \cite{Hatcher}. It is often omitted in further analysis beyond the persistence diagram.

Suppose that a persistence diagram $D$ has $n$ features. For a feature $\xi_i$ with birth $b_i$ and death $d_i$, we construct a function $g_i: \mathbb{R} \rightarrow \mathbb{R}$ as
\begin{equation*}
	g_i(t)= 
	\begin{cases}
		t-b_i & \textrm{if } t \in\left[b_i,\left(b_i+d_i\right) / 2\right], \\ 
		d_i-t & \textrm{if } t \in\left[\left(b_i+d_i\right) / 2, d_i\right], \\ 
		0 & \textrm{otherwise}.
	\end{cases}
\end{equation*}
Next, we define the $j$\textsuperscript{th} layer function $\lambda_j: \mathbb{R} \rightarrow \mathbb{R}$ as
\begin{equation*}
	\lambda_j(t)={\max}^{(j)}\left\{g_1(t), g_2(t), \hdots, g_n(t)\right\}
\end{equation*}
where ${\max}^{(j)}$ denotes the $j$\textsuperscript{th} maximum of the above set. For clarity, we set $\lambda_j(t)=0$ for $j>n$. The vector $\vec{\lambda}=\left(\lambda_1, \lambda_2, \hdots\right)$ is called the \emph{persistence landscape} of $D$. Denote the transformation as $\varphi: D \mapsto \vec{\lambda}$. It is proved in \cite{Betthauser-Bubenik-Edwards} that $\varphi$ is an injection from the set $\mathcal{D}$ of possible persistence diagrams to its image $\varphi(\mathcal{D})$. So, every persistence landscape corresponds to a unique persistence diagram.

We shall note that the set $\varphi(\mathcal{D})$ itself is not a vector space, but it can be embedded into a relevant normed space. For a persistence landscape $\vec{\lambda}$, its \emph{$p$-norm} for $p \geq 1$ and \emph{supremum norm} are given respectively as
\begin{equation*}
	\|\vec{\lambda}\|_p=\left(\sum_{j=1}^{\infty}\left(\int\left|\lambda_j(t)\right|^p d t\right)\right)^{1 / p}
\end{equation*}
and
\begin{equation*}
	\|\vec{\lambda}\|_{\infty}=\sup_{j, t}\left|\lambda_j(t)\right|.
\end{equation*}
For $p=2$, the norm induces an inner product as follows:
\begin{equation*}
	\langle\vec{\lambda}, \vec{\mu}\rangle=\sum_{j=1}^{\infty}\left(\int \lambda_j(t) \cdot \mu_j(t) \ d t\right)
\end{equation*}
where $\vec{\lambda}$ and $\vec{\mu}=\left(\mu_1, \mu_2, \hdots\right)$ are persistence landscapes.

The persistence diagram and landscape are stable; that is, they are robust against a small perturbation onto the point set $P$ \cite{Otter,Bubenik_2}. In such situation, they differ slightly in terms of their distances or norms.

Figure \ref{figure: sample illustration} illustrates the persistence diagram and landscape for a sample from unit disc in $\mathbb{R}^2$. The notations $\mathsf{H}_0$ and $\mathsf{H}_1$ below refer to the respective collection of features with respect to zeroth and first homology groups in the Vietoris-Rips filtration of the sample.
\begin{figure}[H]
	\centering
	\includegraphics[scale=0.72]{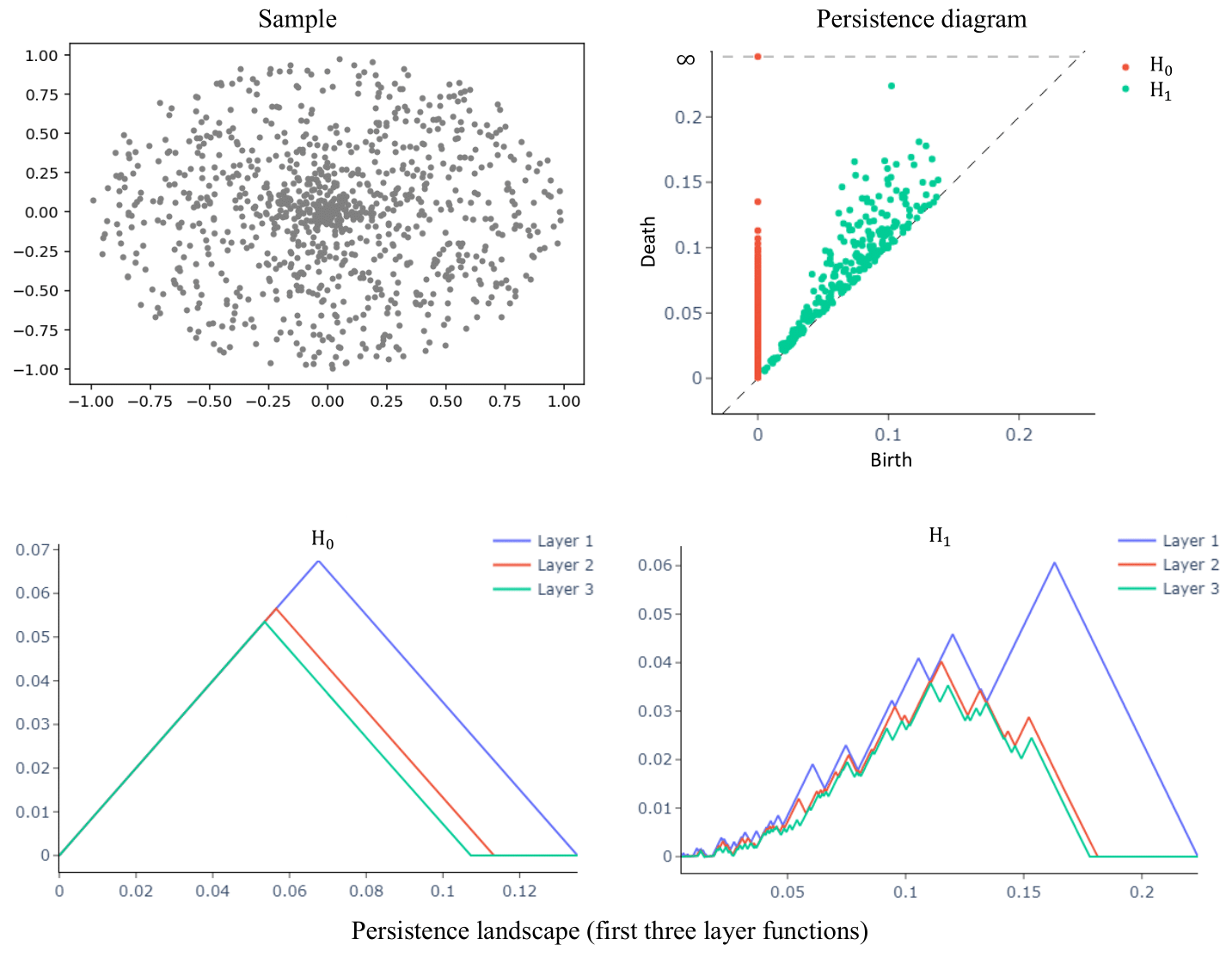}
	\caption{Persistence diagram and landscape for a sample from unit disc in $\mathbb{R}^2$}
	\label{figure: sample illustration}
\end{figure}

\section{Similarity between Persistence Diagrams}\label{section: cosine similarity}
We concentrate on the idea of similarity between persistence diagrams in this section. As mentioned previously, the similarity means how close the points in one diagram to those of another one. We will introduce and explore the properties of cosine similarity as an indicator for the similarity. First, we shall explain why distances may not be an accurate approach for this purpose.

\subsection{Issues with Distances}
Recall that the bottleneck and Wasserstein distances are commonly utilized to indicate the similarity between two persistence diagrams. In particular, these distances measure the degree of difference between two diagrams. The distances defined by the norms of their persistence landscapes can be considered for this purpose as well.

It is generally thought that two diagrams are closely similar whenever their distance is close to zero. However, this assertion is actually false in certain cases. It is possible for two diagrams to be almost equal but their distance is significantly larger than zero. Conversely, two diagrams that are entirely different may record a very small distance. These are illustrated in the following examples.

\begin{example}\label{example: distance not accurate 1}
	Consider two persistence diagrams
	\begin{equation*}
		D_1=\left\{(\theta, \theta+1) \mid \theta=0,1, \ldots, n-1\right\} \quad \textrm{and} \quad D_2=D_1 \cup\left\{(n, n+1)\right\}
	\end{equation*}
	for some $n \in \mathbb{N}$. It is easy to verify the following:
	\begin{gather*}
		W_{\infty}(D_1, D_2)=W_p(D_1, D_2)=\frac{1}{2}, \\
		\left\|\varphi(D_1)-\varphi(D_2)\right\|_{\infty}=\frac{1}{2} \quad \textrm{and} \quad\left\|\varphi(D_1)-\varphi(D_2)\right\|_p=\frac{1}{2}\left(\frac{1}{p+1}\right)^{\frac{1}{p}}.
	\end{gather*}
	The above distances are non-zero and constant for any $n$. We shall expect a significant difference between both diagrams based on the distances. However, they have $n$ common points and differ only by a point. As n increases, the diagrams have a higher proportion of common points and thus are considered almost equal. However, this is not reflected by the distances. This example demonstrates that two diagrams can be very similar but their distances are notably larger than zero.
\end{example}

\begin{example}\label{example: distance not accurate 2}
	Consider two persistence diagrams
	\begin{equation*}
		D_3=\left\{\left.\left(\theta+\frac{1}{2}-\frac{1}{2 m}, \theta+\frac{1}{2}+\frac{1}{2 m}\right) \right\rvert\, \theta=0,1, \ldots, n-1\right\}
	\end{equation*}
	and
	\begin{equation*}
		D_4=\left\{\left.\left(\theta+\frac{1}{2}-\frac{1}{2 m}, \theta+\frac{1}{2}+\frac{1}{2 m}\right) \right\rvert\, \theta=2n, 2n+1, \ldots, 2n+(n-1)\right\}
	\end{equation*}
	for some $m, n \in \mathbb{N}$. Simple calculations show that
	\begin{gather*}
		W_{\infty}(D_3, D_4)=\frac{1}{2 m}, \qquad W_p(D_3, D_4)=\frac{(2 n)^{\frac{1}{p}}}{2 m}, \\
		\|\varphi(D_3)-\varphi(D_4)\|_{\infty}=\frac{1}{2 m} \quad \textrm{and} \quad \|\varphi(D_3)-\varphi(D_4)\|_p=\frac{1}{2 m}\left(\frac{2 n}{m(p+1)}\right)^{\frac{1}{p}}.
	\end{gather*}
	By fixing $n$, the distances approach zero as $m$ increases. We shall expect both diagrams to be closely similar based on the distances. However, for sufficiently large $n$, the points of $D_3$ are clearly very far apart from those of $D_4$ for all $m$. Hence, both diagrams can be considered entirely different. This example shows that two diagrams can be dissimilar but their distances are very small.
\end{example}

\subsection{Cosine Similarity}
The previous argument shows that distances may not be accurate to indicate similarity between persistence diagrams in some cases. Due to this reason, there is a need for a better indicator for the similarity. We introduce the cosine similarity for this purpose. In this section, the notation $\Vert \cdot \rVert$ refers to the $2$-norm for persistence landscapes.

\begin{definition}
	For non-empty persistence diagrams $D_1$ and $D_2$, the \emph{(landscape) cosine similarity} $\varsigma$ is defined as
	\begin{equation*}
		\varsigma(D_1, D_2)=\frac{\left\langle\varphi(D_1), \varphi(D_2)\right\rangle}{\left\|\varphi(D_1)\right\|\left\|\varphi(D_2)\right\|}.
	\end{equation*}
\end{definition}

The assumptions on $D_1$ and $D_2$ are required in the above definition to avoid any zero norm in the denominator. Only empty persistence diagram is mapped to the zero landscape under $\varphi$ whose norm is zero.

We investigate the range of $\varsigma$. It is clear that $\varsigma$ is non-negative. The Cauchy-Schwarz inequality \cite{Beckenbach-Bellman} implies that
\begin{equation}\label{equation: Cauchy-Schwarz}
	\left\langle\varphi(D_1), \varphi(D_2)\right\rangle \leq\left\|\varphi(D_1)\right\|\left\|\varphi(D_2)\right\|.
\end{equation}
This proves that $\varsigma \in[0,1]$. Intuitively from the above inequality, the cosine similarity calculates the proportion of inner product over its possible maximum value. Here, the values $0$ and $1$ correspond to the extreme cases for $\varsigma$, which shall be called the \emph{perfect dissimilarity} and \emph{perfect similarity} between diagrams, respectively. We shall interpret the meaning of these cases in details.

The case $\varsigma(D_1, D_2)=1$ corresponds to the equality in (\ref{equation: Cauchy-Schwarz}). This holds if and only if $\varphi(D_2)=\alpha \cdot \varphi(D_1)$ for a real $\alpha>0$. Based on the definition of a persistence landscape, each layer function is piecewise linear with slope $0$, $1$ or $-1$. This implies $\alpha=1$. Since $\varphi$ is injective, the equation $\varphi(D_2)=\varphi(D_1)$ is equivalent to $D_2=D_1$. Thus, the perfect similarity under $\varsigma$ means simply as the equality of persistence diagrams.

Now, another case $\varsigma(D_1, D_2)=0$ indicates a zero inner product between persistence landscapes of $D_1$ and $D_2$. In this instance, the diagrams shall be referred as orthogonal. We will discuss the notion of orthogonality in Section \ref{section: orthogonality} and justify why it signifies the perfect dissimilarity between persistence diagrams.

The cosine similarity is stable against a small perturbation in persistence diagrams.

\begin{proposition}\label{proposition: continuity of cosine similarity}
	Let $D_1$ and $D_2$ be non-empty persistence diagrams. Fix a real $\eta>0$ such that
	\begin{equation*}
		\eta^2=\min \left\{1, \frac{\left\|\varphi(D_1)\right\|^2}{8}\left(W_{\infty}(D_1, \emptyset)+\frac{1}{3}\right)^{-1}, \frac{\left\|\varphi(D_2)\right\|^2}{8}\left(W_{\infty}(D_2, \emptyset)+\frac{1}{3}\right)^{-1}\right\}.
	\end{equation*}
	Then there are positive constants $c_1$ and $c_2$, which depend on $D_1$ and $D_2$, such that
	\begin{equation*}
		\left|\varsigma(\tilde{D}_1, \tilde{D}_2)-\varsigma(D_1, D_2)\right| \leq c_1 \cdot W_2(D_1, \tilde{D}_1)+c_2 \cdot W_2(D_2, \tilde{D}_2)
	\end{equation*}
	for any pair of non-empty persistence diagrams $\tilde{D}_1$ and $\tilde{D}_2$ with $W_2(D_1, \tilde{D}_1) \leq \eta$ and $W_2(D_2, \tilde{D}_2) \leq \eta$.
\end{proposition}
\begin{proof}
	Recall the equation for an inner product as follows:
	\begin{equation}\label{equation: inspiration variation}
		\left\|\varphi(D_1)-\varphi(D_2)\right\|^2=\left\|\varphi(D_1)\right\|^2+\left\|\varphi(D_2)\right\|^2-2\left\langle\varphi(D_1), \varphi(D_2)\right\rangle.
	\end{equation}
	We rearrange this equation to obtain
	\begin{equation*}
		2 \cdot \varsigma(D_1, D_2)=\frac{\left\|\varphi(D_1)\right\|}{\left\|\varphi(D_2)\right\|}+\frac{\left\|\varphi(D_2)\right\|}{\left\|\varphi(D_1)\right\|}-\frac{\left\|\varphi(D_1)-\varphi(D_2)\right\|^2}{\left\|\varphi(D_1)\right\|\left\|\varphi(D_2)\right\|}.
	\end{equation*}
	The same holds for $\varsigma(\tilde{D}_1, \tilde{D}_2)$. Hence
	\begin{align*}
		2 \left|\varsigma(\tilde{D}_1, \tilde{D}_2)-\varsigma(D_1, D_2)\right| \leq & \left|\frac{\|\varphi(\tilde{D}_1)\|}{\|\varphi(\tilde{D}_2)\|}-\frac{\|\varphi(D_1)\|}{\|\varphi(D_2)\|}\right|
		+\left|\frac{\|\varphi(\tilde{D}_2)\|}{\|\varphi(\tilde{D}_1)\|}-\frac{\|\varphi(D_2)\|}{\|\varphi(D_1)\|}\right| \\
		& \qquad +\left|\frac{{\|\varphi(\tilde{D}_1)-\varphi(\tilde{D}_2)\|}^2}{\|\varphi(\tilde{D}_1)\|\|\varphi(\tilde{D}_2)\|}-\frac{{\|\varphi(D_1)-\varphi(D_2)\|}^2}{\|\varphi(D_1)\|\|\varphi(D_2)\|}\right|.
	\end{align*}
	We shall find a suitable bound for each modulus in the right-hand side (RHS) of the above inequality.
	
	For each $i \in\{1,2\}$, it is proved in \cite{Bubenik_1} that
	\begin{equation}\label{equation: bound for norm and Wasserstein}
		{\|\varphi(\tilde{D}_i)-\varphi(D_i)\|}^2 \leq 2\left(W_{\infty}(D_i, \emptyset)+\frac{1}{3}\right) {W_2(D_i, \tilde{D}_i)}^2
	\end{equation}
	if $W_2(D_i, \tilde{D}_i) \leq 1$. We apply triangle inequality and the assumption on $\eta$ to the above result to obtain
	\begin{equation*}
		\left|\|\varphi(\tilde{D}_i)\|-\|\varphi(D_i)\|\right| \leq \frac{\|\varphi(D_i)\|}{2}
	\end{equation*}
	or equivalently
	\begin{equation}\label{equation: bound for norm}
		\frac{\|\varphi(D_i)\|}{2} \leq\|\varphi(\tilde{D}_i)\| \leq \frac{3\left\|\varphi(D_i)\right\|}{2}.
	\end{equation}
	With these findings, we use triangle inequality, (\ref{equation: bound for norm and Wasserstein}) and (\ref{equation: bound for norm}) to determine bound for the first modulus as follows:
	\begin{align*}
		&\left|\frac{\|\varphi(\tilde{D}_2)\|}{\|\varphi(\tilde{D}_1)\|}-\frac{\|\varphi(D_2)\|}{\|\varphi(D_1)\|}\right| \\
		&\qquad \leq \frac{\|\varphi(D_2)\| \cdot\left|\|\varphi(\tilde{D}_1)\|-\|\varphi(D_1)\right|+\|\varphi(D_1)\| \cdot\left|\|\varphi(\tilde{D}_2)\|-\|\varphi(D_2)\| \right|}{\|\varphi(D_2)\|\|\varphi(\tilde{D}_2)\|} \\
		& \qquad \leq \frac{2\left\|\varphi(D_2)\right\| \cdot\left|\|\varphi(\tilde{D}_1)\|-\|\varphi(D_1)\right|+2\left\|\varphi(D_1)\right\| \cdot\left|\|\varphi(\tilde{D}_2)\|-\|\varphi(D_2)\| \right|}{{\|\varphi(D_2)\|}^2} \\
		& \qquad \leq \frac{2\sqrt{2}}{\|\varphi(D_2)\|}\left(W_{\infty}(D_1, \emptyset)+\frac{1}{3}\right)^{\frac{1}{2}} W_2(D_1, \tilde{D}_1)\\
		& \qquad \qquad \qquad + 
		\frac{2\sqrt{2}\left\|\varphi(D_1)\right\|}{{\|\varphi(D_2)\|}^2}\left(W_{\infty}(D_2, \emptyset)+\frac{1}{3}\right)^{\frac{1}{2}} W_2(D_2, \tilde{D}_2).
	\end{align*}
	Similar result is achieved for the second modulus by symmetry.
	
	We employ the same idea to the third modulus, but it is more cumbersome. Again, we utilize triangle inequality and (\ref{equation: bound for norm}) to lead to
	\begin{align*}
		&\left|\frac{{\|\varphi(\tilde{D}_1)-\varphi(\tilde{D}_2)\|}^2}{\|\varphi(\tilde{D}_1)\|\|\varphi(\tilde{D}_2)\|}-\frac{{\|\varphi(D_1)-\varphi(D_2)\|}^2}{\|\varphi(D_1)\|\|\varphi(D_2)\|}\right| \\
		&\qquad \leq \frac{4}{\|\varphi(D_1)\|\|\varphi(D_2)\|} \cdot \left|{\|\varphi(\tilde{D}_1)-\varphi(\tilde{D}_2)\|}^2 - {\|\varphi(D_1)-\varphi(D_2)\|}^2 \right|\\
		&\qquad \qquad \qquad + \frac{4\left\|\varphi(D_1)-\varphi(D_2)\right\|^2}{\|\varphi(D_1)\|{\|\varphi(D_2)\|}^2} \cdot \left|\|\varphi(\tilde{D}_2)\|- \|\varphi(D_2)\| \right| \\
		&\qquad \qquad \qquad + \frac{4\left\|\varphi(D_1)-\varphi(D_2)\right\|^2}{{\|\varphi(D_1)\|}^2{\|\varphi(D_2)\|}^2} \cdot \|\varphi(\tilde{D}_1)\| \cdot \left|\|\varphi(\tilde{D}_1)\|- \|\varphi(D_1)\| \right|.
	\end{align*}
	It remains to find a bound for each term in the RHS of the above inequality. The aim here is to eliminate the dependency on $\tilde{D}_1$ and $\tilde{D}_2$. So, it is sufficient to consider only the expressions containing these diagrams. In the second and third terms, the bounds for the relevant expressions are obtained easily from (\ref{equation: bound for norm and Wasserstein}) and (\ref{equation: bound for norm}), which are
	\begin{equation*}
		\left|\|\varphi(\tilde{D}_2)\|- \|\varphi(D_2)\| \right| \leq \sqrt{2} \left(W_{\infty}(D_2, \emptyset)+\frac{1}{3}\right)^{\frac{1}{2}} W_2(D_2, \tilde{D}_2)
	\end{equation*}
	and
	\begin{equation*}
		\|\varphi(\tilde{D}_1)\| \cdot \left|\|\varphi(\tilde{D}_1)\|- \|\varphi(D_1)\| \right| \leq \frac{3 \sqrt{2} \left\|\varphi(D_1)\right\|}{2}\left(W_{\infty}(D_1, \emptyset)+\frac{1}{3}\right)^{\frac{1}{2}} W_2(D_1, \tilde{D}_1).
	\end{equation*}
	For the first term, we consider two separate expressions below and then apply triangle inequality, (\ref{equation: bound for norm and Wasserstein}) and (\ref{equation: bound for norm}) as follows:
	\begin{align*}
		& \|\varphi(\tilde{D}_1)-\varphi(\tilde{D}_2)\|+\|\varphi(D_1)-\varphi(D_2)\| \\
		& \qquad \leq \|\varphi(\tilde{D}_1)\|+\|\varphi(\tilde{D}_2)\|+\|\varphi(D_1)-\varphi(D_2)\| \\
		& \qquad \leq \frac{3\left\|\varphi(D_1)\right\|}{2}+\frac{3\left\|\varphi(D_2)\right\|}{2}+\|\varphi(D_1)-\varphi(D_2)\|
	\end{align*}
	and
	\begin{align*}
		& \left|\|\varphi(\tilde{D}_1)-\varphi(\tilde{D}_2)\|-\|\varphi(D_1)-\varphi(D_2)\|\right| \\
		& \qquad \leq \|\varphi(\tilde{D}_1)-\varphi(D_1)\|+\|\varphi(\tilde{D}_2)-\varphi(D_2)\| \\
		& \qquad \leq \sqrt{2}\left(W_{\infty}(D_1, \emptyset)+\frac{1}{3}\right)^{\frac{1}{2}} W_2(D_1, \tilde{D}_1)+\sqrt{2}\left(W_{\infty}(D_2, \emptyset)+\frac{1}{3}\right)^{\frac{1}{2}} W_2(D_2, \tilde{D}_2).
	\end{align*}
	Overall, we combine all relevant inequalities to calculate the constants $c_1$ and $c_2$.
\end{proof}

Although not required here, the previous proposition implies that $\varsigma$ is a continuous map. For this purpose, we equip the set $\mathcal{D}$ of persistence diagrams with $2$-Wasserstein distance. We consider the product metric on $\mathcal{D}^2$. The continuity of $\varsigma$ is proved below.

\begin{corollary}
	The map $\varsigma:\left(\mathcal{D} \setminus \{\emptyset\}\right)^2 \rightarrow \mathbb{R}$ is continuous.
\end{corollary}
\begin{proof}
	Let $D_1$ and $D_2$ be non-empty persistence diagrams. Consider the positive real $\eta, c_1$ and $c_2$ from Proposition 1. Fix a real $\epsilon>0$. Set
	\begin{equation*}
		\delta=\min \left\{\eta, \frac{\epsilon}{2 c_1}, \frac{\epsilon}{2 c_2}\right\}.
	\end{equation*}
	The inequality
	\begin{equation*}
		\sqrt{{W_2(D_1, \tilde{D}_1)}^2+{W_2(D_2, \tilde{D}_2)}^2}<\delta
	\end{equation*}
	implies that $W_2(D_1, \tilde{D}_1) \leq \eta$ and $W_2(D_2, \tilde{D}_2) \leq \eta$. So, the conclusion of Proposition \ref{proposition: continuity of cosine similarity} holds. Together with the assumption on $\delta$, we obtain
	\begin{equation*}
		|\varsigma(\tilde{D}_1, \tilde{D}_2)-\varsigma(D_1, D_2)| \leq c_1 \cdot W_2(D_1, \tilde{D}_1)+c_2 \cdot W_2(D_2, \tilde{D}_2)<\epsilon.
	\end{equation*}
	This shows the continuity of $\varsigma$.
\end{proof}

For completeness, we introduce cosine distance as counterpart to cosine similarity. We will use this notion for demonstration in Section \ref{section: data demo} as a direct comparison to other distances between persistence diagrams. We shall emphasize that the term itself is a misnomer since it is not a metric, but has been named this way in data science \cite{Tan}.

\begin{definition}
	For non-empty persistence diagrams $D_1$ and $D_2$, the \emph{(landscape) cosine distance} $\varsigma^*$ is defined as
	\begin{equation*}
		\varsigma^*(D_1,D_2) = 1 - \varsigma(D_1,D_2).
	\end{equation*}
\end{definition}

\begin{example}[Examples \ref{example: distance not accurate 1} and \ref{example: distance not accurate 2} revisited]
	We can verify that
	\begin{equation*}
		\varsigma(D_1, D_2)=\frac{n}{\sqrt{n(n+1)}}.
	\end{equation*}
	This value approaches $1$ as $n$ increases, which indicates that $D_1$ and $D_2$ become almost equal. This agrees with our direct observation on both diagrams.
	
	On the other hand, we calculate that $\varsigma(D_3,D_4)=0$ for any $m$ and $n$, which implies that $D_3$ and $D_4$ are perfectly dissimilar. This fits our previous explanation that their collections of points are far apart from one another and thus, both diagrams deserve to be considered entirely different.
\end{example}

\section{Orthogonality between Persistence Diagrams}\label{section: orthogonality}
According to cosine similarity, two persistence diagrams are perfectly dissimilar if they are orthogonal. We shall justify why the notion of orthogonality is suitable to describe the perfect dissimilarity between persistence diagrams.

\begin{definition}
	Persistence diagrams $D_1$ and $D_2$ are said to be \emph{orthogonal} if
	\begin{equation*}
		\langle \varphi(D_1), \varphi(D_2) \rangle = 0.
	\end{equation*}
\end{definition}

The above definition of orthogonality relies on the idea of persistence landscapes. We shall provide an equivalent definition which concentrates on persistence diagrams themselves. To avoid confusion below, the notation $(b,d)$ refers to a point in a diagram while $\left(b,d\right)_{\mathbb{R}}$ denotes an open interval of $\mathbb{R}$.

\begin{proposition}\label{proposition: orthogonal no intersection}
	The persistence diagrams $D_1$ and $D_2$ are orthogonal if and only if for any pair of points $(b_1, d_1) \in D_1$ and $(b_2, d_2) \in D_2$, the intervals $\left(b_1, d_1\right)_{\mathbb{R}}$ and $\left(b_2, d_2\right)_{\mathbb{R}}$ have an empty intersection.
\end{proposition}
\begin{proof}
	For $i \in\{1,2\}$, define
	\begin{equation*}
		L_i=\bigcup_{(b, d) \in D_i}\left(b, d\right)_{\mathbb{R}}.
	\end{equation*}
	Denote $\varphi(D_i)_j$ as the $j$\textsuperscript{th} layer function of persistence landscape $\varphi(D_i)$. Denote also $\operatorname{supp}\left(\varphi(D_i)_j\right)$ as its support i.e. the set of values $t \in \mathbb{R}$ where $\varphi(D_i)_j(t)>0$. It is clear that
	\begin{equation*}
		\operatorname{supp}\left(\varphi(D_i)_1\right)=L_i \quad \textrm{and} \quad \operatorname{supp}\left(\varphi(D_i)_j\right) \subseteq L_i
	\end{equation*}	
	for $j>1$.
	
	Suppose that $D_1$ and $D_2$ are orthogonal. We have
	\begin{equation*}
		\int \varphi(D_1)_1(t) \cdot \varphi(D_2)_1(t) \ dt \leq\langle\varphi\left(D_1\right), \varphi\left(D_2\right)\rangle=0.
	\end{equation*}
	Each layer function above is continuous according to \cite{Bubenik_1}, and so is their product. Since the integral is zero, the function $\varphi(D_1)_1 \cdot \varphi(D_2)_2$ is zero everywhere. This shows that
	\begin{equation*}
		L_1 \cap L_2=\operatorname{supp}\left(\varphi(D_1)_1\right) \cdot \operatorname{supp}\left(\varphi(D_2)_1\right)=\operatorname{supp}\left(\varphi(D_1)_1 \cdot \varphi(D_2)_1\right)=\emptyset
	\end{equation*}
	as desired.
	
	Conversely, suppose that $L_1 \cap L_2=\emptyset$. Note that
	\begin{equation*}
		\operatorname{supp}\left(\varphi(D_1)_j \cdot \varphi(D_2)_j\right)=\operatorname{supp}\left(\varphi(D_1)_j\right) \cdot \operatorname{supp}\left(\varphi(D_2)_j\right) \subseteq L_1 \cap L_2=\emptyset
	\end{equation*}
	for any $j$. This implies that $\varphi(D_1)_j \cdot \varphi(D_2)_j$ is a zero function and hence, we conclude that $\langle\varphi(D_1), \varphi(D_2)\rangle=0$.
\end{proof}

\begin{corollary}
	If persistence diagrams $D_1$ and $D_2$ are orthogonal, then
	\begin{equation*}
		\|\varphi(D_1)-\varphi(D_2)\|_{\infty}=\sup \big\{\|\varphi(D_1)\|_{\infty},\|\varphi(D_2)\|_{\infty}\big\}
	\end{equation*}
	and
	\begin{equation*}
		{\|\varphi(D_1)-\varphi(D_2)\|_p}^p={\|\varphi(D_1)\|_p}^p + {\|\varphi(D_2)\|_p}^p.
	\end{equation*}
\end{corollary}
\begin{proof}
	This is straightforward by definitions of the norms and also the implication of Proposition \ref{proposition: orthogonal no intersection}.
\end{proof}

The previous proposition gives an intuitive meaning of orthogonality for persistence diagrams from the view of their simplicial filtrations. Given two orthogonal diagrams, whenever a feature exists during a certain range of filtration parameter for a diagram, there is none for another diagram in the same range. Due to this behavior, their ground simplicial complexes can be seen as a clear contrast to each other, and thus the same can be said to the persistence diagrams as well. This is indeed a sound reason on why the orthogonality symbolizes a perfect dissimilarity between diagrams, though we have a second justification below.

Recall the trivial matching based on the definitions of bottleneck and Wasserstein distances in Section \ref{section: persistent homology}. In general, the bijection is considered as a bad pairing of points between two persistence diagrams. It disregards the pairs of points in these diagrams who are closer to each other compared to points on the diagonal line. It is commonly expected that trivial matching is not perfect. However, it is indeed a perfect matching if the diagrams are orthogonal.

\begin{proposition}\label{proposition: orthogonality trivial perfect}
	If persistence diagrams $D_1$ and $D_2$ are orthogonal, then their trivial matching is perfect under the bottleneck and Wasserstein distances.
\end{proposition}
\begin{proof}
	This clearly holds if either $D_1$ and $D_2$ is empty. We assume that both diagrams are non-empty. For ease of notation, define
	\begin{equation*}
		\omega_{\infty}(\gamma)=\sup _{x \in D_1^{\Delta}}\|x-\gamma(x)\|_{\infty} \quad \textrm{and} \quad \omega_p(\gamma)=\sum_{x \in D_1^{\Delta}}{\|x-\gamma(x)\|_\infty}^p
	\end{equation*}
	for a bijection $\gamma: D_1^{\Delta} \rightarrow D_2^{\Delta}$ in the definitions of bottleneck and Wasserstein distances. Denote $\kappa$ as the trivial matching. For points $x=(b_1, d_1) \in D_1$ and $y=(b_2, d_2) \in D_2$, their paired points on the diagonal line $\Delta$ are
	\begin{equation*}
		\kappa(x)=\left(\frac{b_1+d_1}{2}, \frac{b_1+d_1}{2}\right) \quad \textrm{and} \quad \kappa^{-1}(y)=\left(\frac{b_2+d_2}{2}, \frac{b_2+d_2}{2}\right).
	\end{equation*}
	Note that the points above are the closest to $x$ and $y$ respectively among points on $\Delta$ with respect to supremum norm in $\mathbb{R}^2$.
	
	Suppose that $D_1$ and $D_2$ are orthogonal. We claim that
	\begin{equation}\label{equation: inequality sup norm}
		\sup \big\{\|x-\kappa(x)\|_{\infty},\|y-\kappa^{-1}(y)\|_{\infty}\big\} \leq\|x-y\|_{\infty}
	\end{equation}
	and
	\begin{equation}\label{equation: inequality p-norm}
		{\|x-\kappa(x)\|_\infty}^p+{\|y-\kappa^{-1}(y)\|_\infty}^p \leq {\|x-y\|_\infty}^p.
	\end{equation}
	Proposition \ref{proposition: orthogonal no intersection} states that $\left(b_1, d_1\right)_{\mathbb{R}} \cap\left(b_2, d_2\right)_{\mathbb{R}}=\emptyset$. This implies that either $d_1 \leq b_2$ or $d_2 \leq b_1$. Without loss of generality, we assume the former. The above inequalities are proved respectively as follows:
	\begin{align*}
		\sup \left\{\frac{d_1-b_1}{2}, \frac{d_2-b_2}{2}\right\} & <\sup \left\{d_1-b_1, d_2-b_2\right\} \\
		& \leq \sup \left\{b_2-b_1, d_2-d_1\right\}
	\end{align*}
	and
	\begin{align*}
		2\left(\frac{d_1-b_1}{2}\right)^p+2\left(\frac{d_2-b_2}{2}\right)^p & \leq\left(d_1-b_1\right)^p+\left(d_2-b_2\right)^p \\
		& \leq\left(b_2-b_1\right)^p+\left(d_2-d_1\right)^p\\
		& \leq 2 \cdot \sup\left\{\left(b_2-b_1\right)^p, \left(d_2-d_1\right)^p\right\}.
	\end{align*}
	Consider a bijection $\gamma: D_1^{\Delta} \rightarrow D_2^{\Delta}$. For a point $x \in D_1$, note that if $\gamma(x) \in \Delta$, then
	\begin{equation*}
		\|x-\kappa(x)\|_{\infty} \leq\|x-\gamma(x)\|_{\infty}
	\end{equation*}
	since $\kappa(x)$ is the closest to $x$ among points on $\Delta$. If $\gamma(x)=y \in D_2$, then the inequalities (\ref{equation: inequality sup norm}) and (\ref{equation: inequality p-norm}) hold. These findings conclude that
	\begin{equation*}
		\omega_{\infty}(\kappa) \leq \omega_{\infty}(\gamma) \quad \text { and } \quad \omega_p(\kappa) \leq \omega_p(\gamma)
	\end{equation*}
	for any bijection $\gamma$, and thus $\kappa$ is a perfect matching.
\end{proof}

\begin{corollary}
	If persistence diagrams $D_1$ and $D_2$ are orthogonal, then
	\begin{equation*}
		W_{\infty}(D_1, D_2)=\sup \big\{W_{\infty}(D_1, \emptyset), W_{\infty}(D_2, \emptyset)\big\}
	\end{equation*}
	and
	\begin{equation*}
		{W_p(D_1, D_2)}^p={W_p(D_1, \emptyset)}^p+{W_p(D_2, \emptyset)}^p.
	\end{equation*}
\end{corollary}
\begin{proof}
	This is straightforward by definitions of the distances and also the implication of Proposition \ref{proposition: orthogonality trivial perfect}.
\end{proof}

The last proposition provides an intuitive meaning for orthogonality of persistence diagrams with respect to the distributions of their points on the same plot. The points in a diagram are sufficiently far apart from those in another one to avoid better pairing than the trivial matching. It can be thought that the diagrams occupy separate regions on the plot, and they deserve to be considered entirely different. This provides another justification on why orthogonality signifies a perfect dissimilarity between persistence diagrams.

As a remark, although two orthogonal diagrams are deemed entirely different based on both justifications above, it disagrees with interpretation of distances. Recall that a distance indicates degree of difference. So, it is expected that two orthogonal diagrams have a significantly large distance. However, they can record a very small distance, as illustrated in Example \ref{example: distance not accurate 2}. 

\section{Data Demonstration}\label{section: data demo}

We shall demonstrate the accuracy of cosine similarity against the standard distances to indicate similarity between persistence diagrams. Here, the cosine distance is used instead of cosine similarity to match the narrative with other distances. We will show that the cosine distance is able to distinguish clearly whether two diagrams are closely similar or dissimilar.

\subsection{Data Preparation}

To prepare persistence diagrams in this demonstration, several data sets are sampled from certain regions or shapes in $\mathbb{R}^2$. The sets are listed as follows:
\begin{enumerate}[(i)]
\item $Q$ and $Q'$ from the unit disc $\mathcal{Q}=\left\{x \in \mathbb{R}^2 \mid\|x\|_2 \leq 1\right\}$,
\item $R$ and $R'$ from an annulus $\mathcal{R}=\left\{x \in \mathbb{R}^2 \mid 0.5 \leq\|x\|_2 \leq 1\right\}$, and
\item $S$ and $S'$ from the unit circle $\mathcal{S}=\left\{x \in \mathbb{R}^2 \mid\|x\|_2=1\right\}$.
\end{enumerate}

Each data set contains $3000$ points which are chosen randomly and uniformly from the designated shape. Persistence diagrams for the sets are obtained via Vietoris-Rips filtration. We observe features in zeroth and first homology groups for these diagrams, which are denoted as $\mathsf{H}_0$ and $\mathsf{H}_1$, respectively. The diagrams are then compared against each other via the bottleneck distance $W_{\infty}$, Wasserstein distance $W_2$, distances defined by the norms $\lVert\cdot\rVert_{\infty}$ and $\lVert\cdot\rVert_2$ for persistence landscapes, and lastly, cosine distance $\varsigma^*$. The persistent homology method is carried out via a Python package called giotto-tda \cite{Tauzin}.

For each data set, its points are selected in the above manner so that the filtration is able to capture the topological structure of the shape accurately. Given a parameter value, its Vietoris-Rips complex can be regarded as a reconstruction of the said shape from the data set. Specifically, each $1$-simplex in this complex creates an edge between two points, while each $2$-simplex forms a triangular region enclosed by three points in the data set. The union of these edges and triangles approximates the original shape. It consists of connected components and holes which are specified by persistent diagrams with respect to $\mathsf{H}_0$ and $\mathsf{H}_1$. Any geometrical feature in each shape, such as holes of $\mathcal{R}$ and $\mathcal{S}$, will be presented in the diagrams of corresponding data sets.

\subsection{Objectives}
In this demonstration, we shall observe the following:
\begin{enumerate}[(i)]
	\item distances between persistence diagrams with respect to $\mathsf{H}_0$ and $\mathsf{H}_1$ for three pairs $Q$ and $Q'$, $R$ and $R'$, and $S$ and $S'$, and
	\item distances between persistence diagrams with respect to $\mathsf{H}_1$ among $Q$, $R$ and $S$.
\end{enumerate}
Specifically, we examine whether each distance can convey the similarity between those diagrams clearly. The objectives are decided this way because in theory, we have prior knowledge on whether those diagrams are closely similar or very dissimilar. So, we can compare this theoretical information with the results obtained in this demonstration with confidence.

For the first objective, two data sets are set up for each shape to demonstrate close similarity between their persistence diagrams. Since both are sampled from the same shape randomly and uniformly, their diagrams must have a slight difference only. So, their distances shall be small, including the cosine distance.

The second objective is done to verify extent of similarity and dissimilarity between persistence diagrams with respect to $\mathsf{H}_1$ for the relevant data sets. Recall that features in $\mathsf{H}_1$ relate to holes formed during a simplicial filtration. Each of the shapes $\mathcal{R}$ and $\mathcal{S}$ has a hole, unlike $\mathcal{Q}$. This hole will be recorded as a persistent feature in the diagrams with respect to $\mathsf{H}_1$ for data sets $R$ and $S$. Such feature will be absent in the diagram for $Q$. Thus, the diagram with respect to $\mathsf{H}_1$ for $R$ shall be similar to that for $S$, but dissimilar to that for $Q$. The cosine distance is expected to address this similarity or dissimilarity clearly in comparison to other distances.

\subsection{Results}

Figures \ref{figure: persistence diagrams Q} to \ref{figure: persistence diagrams S} show the persistence diagrams for the data sets. The distances between relevant diagrams are summarized in Tables \ref{table: same shape} to \ref{table: cosine distance}.

\begin{figure}[H]
	\centering
	\includegraphics[scale=0.7]{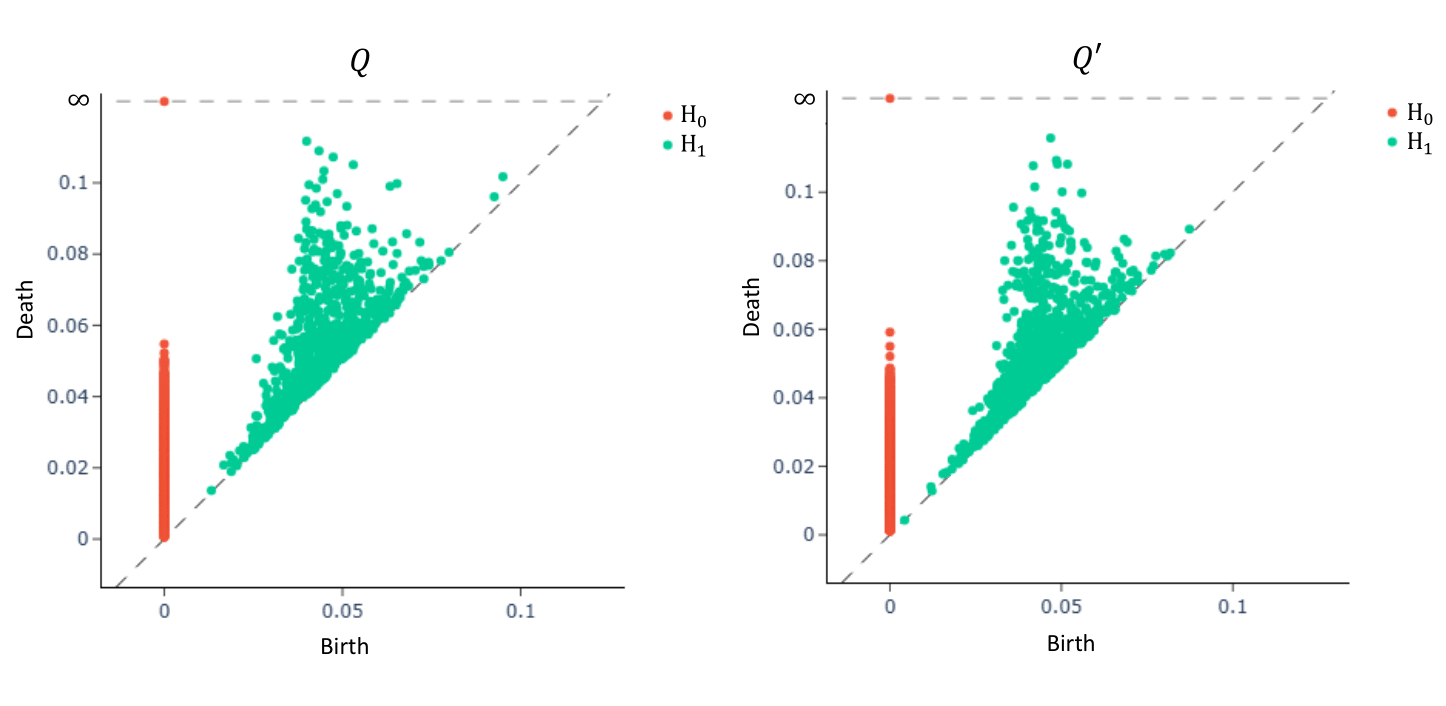}
	\caption{Persistence diagrams for data sets $Q$ and $Q'$}
	\label{figure: persistence diagrams Q}
\end{figure}
\begin{figure}[H]
	\centering
	\includegraphics[scale=0.7]{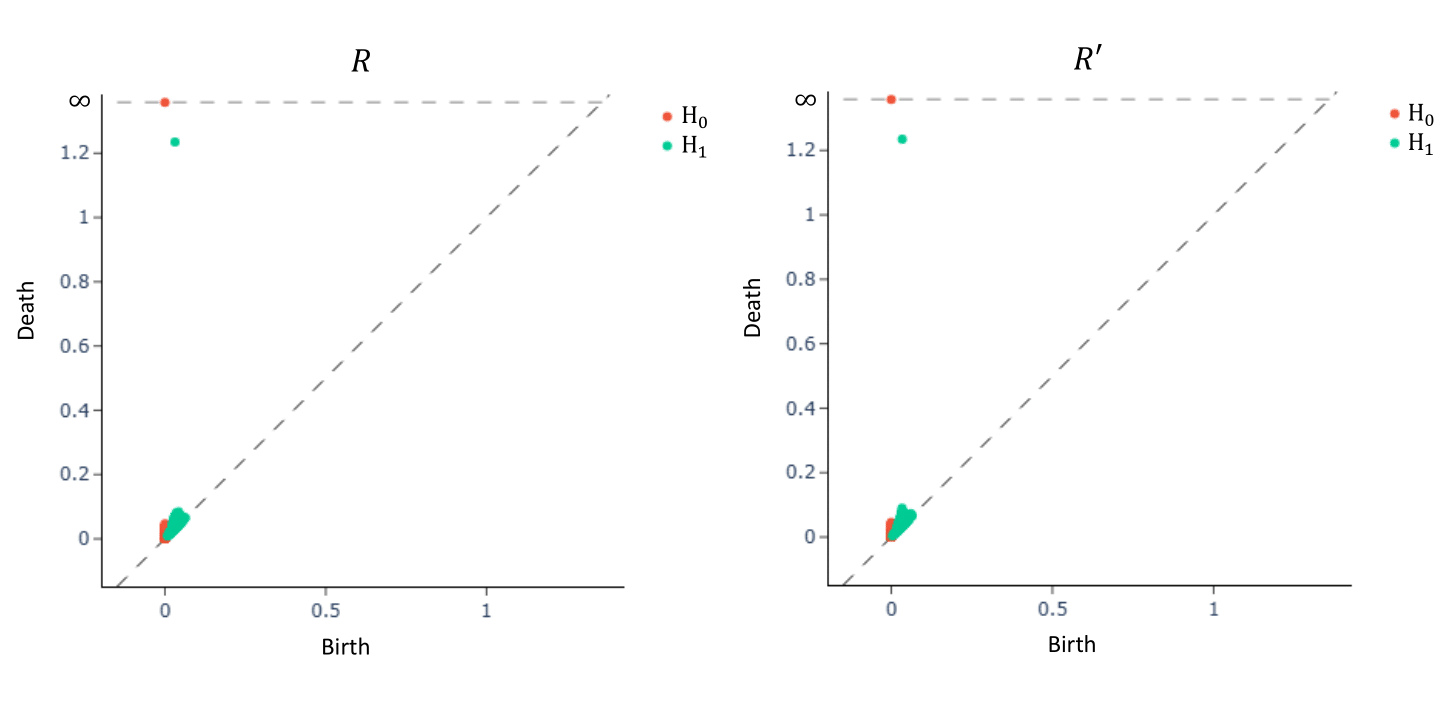}
	\caption{Persistence diagrams for data sets $R$ and $R'$}
	\label{figure: persistence diagrams R}
\end{figure}
\begin{figure}[H]
	\centering
	\includegraphics[scale=0.7]{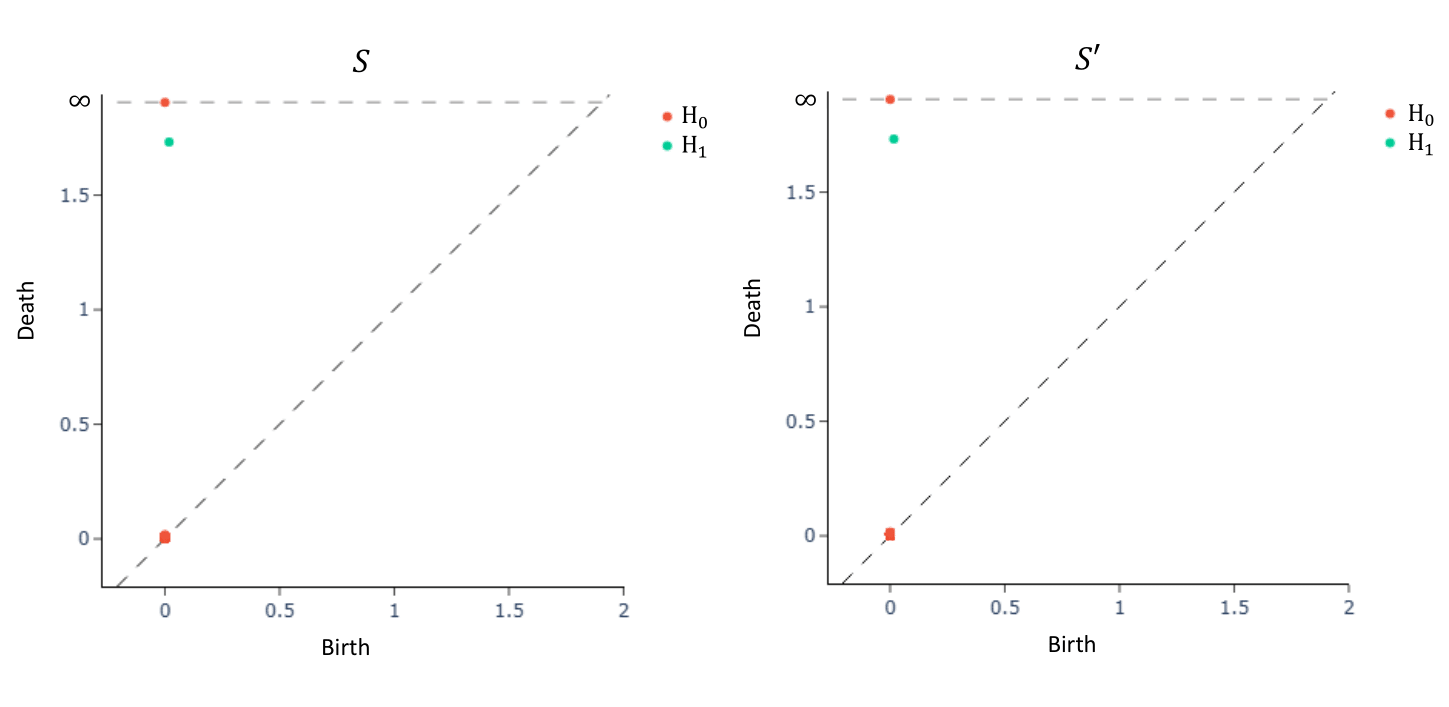}
	\caption{Persistence diagrams for data sets $S$ and $S'$}
	\label{figure: persistence diagrams S}
\end{figure}
\newpage
\begin{table}[h]
	\centering
	\caption{Distances between persistence diagrams for two data sets from the same shape}
	\begin{tabular}{@{}clllll@{}}
		\multicolumn{6}{c}{$\mathsf{H}_0$}                                                                                                                                                       \\ \toprule
		Data sets    & \multicolumn{1}{c}{$W_\infty$} & \multicolumn{1}{c}{$W_2$} & \multicolumn{1}{c}{$\|\cdot\|_\infty$} & \multicolumn{1}{c}{$\|\cdot\|_2$} & \multicolumn{1}{c}{$\varsigma^*$} \\ \midrule
		$Q$ and $Q'$ & 0.004360                       & 0.015458                  & 0.004360                               & 0.000722                          & 0.010417                        \\
		$R$ and $R'$ & 0.001629                       & 0.013998                  & 0.001629                               & 0.000245                          & 0.002320                        \\
		$S$ and $S'$ & 0.001921                       & 0.006393                  & 0.001921                               & 0.000165                          & 0.020274                        \\ \bottomrule
		\multicolumn{6}{c}{} 
	\end{tabular}
	\begin{tabular}{@{}clllll@{}}
		\multicolumn{6}{c}{$\mathsf{H}_1$}                                                                                                                                                       \\ \toprule
		Data sets    & \multicolumn{1}{c}{$W_\infty$} & \multicolumn{1}{c}{$W_2$} & \multicolumn{1}{c}{$\|\cdot\|_\infty$} & \multicolumn{1}{c}{$\|\cdot\|_2$} & \multicolumn{1}{c}{$\varsigma^*$} \\ \midrule
		$Q$ and $Q'$ & 0.010991                       & 0.041436                  & 0.005256                               & 0.001085                          & 0.014863                        \\
		$R$ and $R'$ & 0.008980                       & 0.034152                  & 0.008365                               & 0.002123                          & 0.000009                        \\
		$S$ and $S'$ & 0.002066                       & 0.002066                  & 0.002066                               & 0.001911                          & 0.000003                        \\ \bottomrule
	\end{tabular}
	\label{table: same shape}
\end{table}

\begin{table}[h]
	\centering
	\caption{Bottleneck distance between persistence diagrams for the data sets}
	\begin{tabular}{@{}ll@{}}
		\begin{tabular}{@{}clll@{}}
			\multicolumn{4}{c}{$\mathsf{H}_0$}                                                \\ \toprule
			& \multicolumn{1}{c}{$Q$} & \multicolumn{1}{c}{$R$} & \multicolumn{1}{c}{$S$} \\ \midrule
			$Q$ & 0                       & 0.013372                & 0.027399                \\
			$R$ & 0.013372                & 0                       & 0.023379                \\
			$S$ & 0.027399                & 0.023379                & 0                       \\ \bottomrule
		\end{tabular}	& 
		\begin{tabular}{@{}clll@{}}
			\multicolumn{4}{c}{$\mathsf{H}_1$}                                                \\ \toprule
			& \multicolumn{1}{c}{$Q$} & \multicolumn{1}{c}{$R$} & \multicolumn{1}{c}{$S$} \\ \midrule
			$Q$ & 0                       & 0.601816                & 0.856912                \\
			$R$ & 0.601816                & 0                       & 0.497006                \\
			$S$ & 0.856912                & 0.497006                & 0                       \\ \bottomrule
		\end{tabular}
	\end{tabular}
	\label{table: bottleneck distance}
\end{table}

\begin{table}[h]
	\centering
	\caption{$2$-Wasserstein distance between persistence diagrams for the data sets}
	\begin{tabular}{@{}ll@{}}
		\begin{tabular}{@{}clll@{}}
			\multicolumn{4}{c}{$\mathsf{H}_0$}                                                \\ \toprule
			& \multicolumn{1}{c}{$Q$} & \multicolumn{1}{c}{$R$} & \multicolumn{1}{c}{$S$} \\ \midrule
			$Q$ & 0                       & 0.352694                & 0.632645                \\
			$R$ & 0.352694                & 0                       & 0.452928                \\
			$S$ & 0.632645                & 0.452928                & 0                       \\ \bottomrule
		\end{tabular}
		&
		\begin{tabular}{@{}clll@{}}
			\multicolumn{4}{c}{$\mathsf{H}_1$}                                                \\ \toprule
			& \multicolumn{1}{c}{$Q$} & \multicolumn{1}{c}{$R$} & \multicolumn{1}{c}{$S$} \\ \midrule
			$Q$ & 0                       & 0.635436                & 0.887606                \\
			$R$ & 0.635436                & 0                       & 0.518727                \\
			$S$ & 0.887606                & 0.518727                & 0                       \\ \bottomrule
		\end{tabular}
	\end{tabular}
	\label{table: Wasserstein distance}
\end{table}
\newpage
\begin{table}[h]
	\centering
	\caption{Distance defined by supremum norm between persistence landscapes for the data sets}
	\begin{tabular}{@{}ll@{}}
		\begin{tabular}{@{}clll@{}}
			\multicolumn{4}{c}{$\mathsf{H}_0$}                                                \\ \toprule
			& \multicolumn{1}{c}{$Q$} & \multicolumn{1}{c}{$R$} & \multicolumn{1}{c}{$S$} \\ \midrule
			$Q$ & 0                       & 0.013372                & 0.027394                \\
			$R$ & 0.013372                & 0                       & 0.023376                \\
			$S$ & 0.027394                & 0.023376                & 0                       \\ \bottomrule
		\end{tabular}
		&
		\begin{tabular}{@{}clll@{}}
			\multicolumn{4}{c}{$\mathsf{H}_1$}                                                \\ \toprule
			& \multicolumn{1}{c}{$Q$} & \multicolumn{1}{c}{$R$} & \multicolumn{1}{c}{$S$} \\ \midrule
			$Q$ & 0                       & 0.601763                & 0.856864                \\
			$R$ & 0.601763                & 0                       & 0.497006                \\
			$S$ & 0.856864                & 0.497006                & 0                       \\ \bottomrule
		\end{tabular}
	\end{tabular}
	\label{table: supremum distance}
\end{table}

\begin{table}[h]
	\centering
	\caption{Distance defined by $2$-norm between persistence landscapes for the data sets}
	\begin{tabular}{@{}ll@{}}
		\begin{tabular}{@{}clll@{}}
			\multicolumn{4}{c}{$\mathsf{H}_0$}                                                \\ \toprule
			& \multicolumn{1}{c}{$Q$} & \multicolumn{1}{c}{$R$} & \multicolumn{1}{c}{$S$} \\ \midrule
			$Q$ & 0                       & 0.001230                & 0.003578                \\
			$R$ & 0.001230                & 0                       & 0.002756                \\
			$S$ & 0.003578                & 0.002756                & 0                       \\ \bottomrule
		\end{tabular}
		&
		\begin{tabular}{@{}clll@{}}
			\multicolumn{4}{c}{$\mathsf{H}_1$}                                                \\ \toprule
			& \multicolumn{1}{c}{$Q$} & \multicolumn{1}{c}{$R$} & \multicolumn{1}{c}{$S$} \\ \midrule
			$Q$ & 0                       & 0.381141                & 0.647616                \\
			$R$ & 0.381141                & 0                       & 0.387817                \\
			$S$ & 0.647616                & 0.387817                & 0                       \\ \bottomrule
		\end{tabular}
	\end{tabular}
	\label{table: norm distance}
\end{table}

\begin{table}[h]
	\centering
	\caption{Cosine distance between persistence diagrams for the data sets}
	\begin{tabular}{@{}ll@{}}
		\begin{tabular}{@{}clll@{}}
			\multicolumn{4}{c}{$\mathsf{H}_0$}                                                \\ \toprule
			& \multicolumn{1}{c}{$Q$} & \multicolumn{1}{c}{$R$} & \multicolumn{1}{c}{$S$} \\ \midrule
			$Q$ & 0                       & 0.041416                & 0.725444                \\
			$R$ & 0.041416                & 0                       & 0.651488                \\
			$S$ & 0.725444                & 0.651488                & 0                       \\ \bottomrule
		\end{tabular}
		&
		\begin{tabular}{@{}clll@{}}
			\multicolumn{4}{c}{$\mathsf{H}_1$}                                                \\ \toprule
			& \multicolumn{1}{c}{$Q$} & \multicolumn{1}{c}{$R$} & \multicolumn{1}{c}{$S$} \\ \midrule
			$Q$ & 0                       & 0.973207                & 0.979322                \\
			$R$ & 0.973207                & 0                       & 0.160770                \\
			$S$ & 0.979322                & 0.160770                & 0                       \\ \bottomrule
		\end{tabular}
	\end{tabular}
	\label{table: cosine distance}
\end{table}

\subsection{Discussion}
In Figures \ref{figure: persistence diagrams Q} to \ref{figure: persistence diagrams S}, each persistence diagram with respect to $\mathsf{H}_0$ has exactly one feature that never dies, as explained in Section \ref{section: persistent homology}. It is omitted in the calculations for distances. Notice that there is a persistent feature in the diagram for $R$ with respect to $\mathsf{H}_1$, which reflects the hole of annulus $\mathcal{R}$. Such feature is also observed for $R'$, $S$ and $S'$ due to a similar reason.

The persistence diagrams for $Q$ and $Q'$ in Figure \ref{figure: persistence diagrams Q} seem to be closely similar. This is indeed confirmed by all distances in Table \ref{table: same shape}, where they are noticeably small. This is expected since the data sets originate from the same shape $\mathcal{Q}$. This is true for other pairs of data sets as well. From this finding, the cosine distance is able to indicate close similarity between two persistence diagrams by recording a small value. However, this is nothing special since other distances conclude the same too. It is worth noting that the performance of cosine distance against other distances varies throughout the pairs of data sets. For example in the pair $S$ and $S'$, the cosine distance records a higher value than other distances for $\mathsf{H_0}$, but it is the opposite for $\mathsf{H_1}$. Nonetheless, the values given by the cosine distance are small enough to ascertain the close similarity between diagrams for each pair of data sets.

Next, we examine the persistence diagrams with respect to $\mathsf{H}_1$ for the data sets $Q$, $R$ and $S$. Note that the persistent feature for $R$ and $S$ is located way above the diagonal line in Figures \ref{figure: persistence diagrams R} and \ref{figure: persistence diagrams S}. Meanwhile, the features in the diagram of $Q$ accumulate around the diagonal line in Figure \ref{figure: persistence diagrams Q}. This contributes a large difference between the diagram of each $R$ and $S$ and that of $Q$. Thus, we establish that the diagram of $R$ is similar to that of $S$, while dissimilar to that of $Q$. This fact is confirmed by the cosine distance in Table \ref{table: cosine distance}. It is small for $R$ and $S$ to indicate their similarity, while it is close to $1$ for $Q$ and $R$ to ascertain their dissimilarity.

However, the same fact is not conveyed clearly by the other distances. Each of their values for $R$ and $S$ is noticeably close to that for $Q$ and $R$. This is especially apparent for the distance defined by $\lVert\cdot\rVert_2$ in Table \ref{table: norm distance}. Hence in this case, these distances cannot distinguish similarity or dissimilarity between persistence diagrams.

It is interesting to note that each distance records a relatively high value for $Q$ and $S$, which implies a large difference between their diagrams with respect to $\mathsf{H}_1$. This is expected since their diagrams are dissimilar. However, this fact is clearly indicated by the cosine distance whose value is close to $1$, and even around the value for $Q$ and $R$. Unlike other distances, it is more confident to distinguish the similarity or dissimilarity between the diagrams for those three data sets by using the cosine distance.

As a conclusion, our demonstration shows that the cosine distance, and thus cosine similarity, are more accurate to indicate similarity between the diagrams in comparison to other distances. We shall clarify that the demonstration has been replicated multiple times with different samples from those three shapes, and our conclusion on the cosine similarity still holds in each run.

\section{Conclusion}
In this paper, we have introduced the cosine similarity as a new indicator for similarity between persistence diagrams. It measures the extent of similarity within two extremes, which are perfect similarity and dissimilarity. While the former simply means equality of two diagrams, the perfect dissimilarity refers to the notion of orthogonality between the diagrams. The cosine similarity is found to be more accurate in this purpose, unlike the common distances for persistence diagrams such as the bottleneck and Wasserstein distances.

For future work, we provide some insights here on the research problem of similarity between persistence diagrams. As mentioned in Section \ref{section: introduction}, there are other approaches in the literature to measure the similarity. We are yet to compare them with our cosine similarity in terms of accuracy and efficiency. For example, the RST parametric model \cite{Agami} is able to account for geometrical and topological differences on the shapes of data sets. This cannot be achieved by the cosine similarity since data sets from two shapes of different sizes but topologically equal are perceived to be different under the cosine similarity. More research is needed to investigate the advantages and disadvantages of each method, including the cosine similarity itself.

It is possible to propose a variation of the cosine similarity. For example, we define another variation as
\begin{equation*}
	\varrho(D_1, D_2)=\frac{2\left\langle\varphi(D_1), \varphi(D_2)\right\rangle}{\left\|\varphi(D_1)\right\|^2+\left\|\varphi(D_2)\right\|^2}
\end{equation*}
for persistence diagrams $D_1$ and $D_2$ where one of them is non-empty. This is motivated by (\ref{equation: inspiration variation}). In fact, this equation enables the corresponding cosine distance to be expressed nicely as
\begin{equation*}
	\varrho^*(D_1, D_2)=\frac{\left\|\varphi(D_1)-\varphi(D_2)\right\|^2}{\left\|\varphi(D_1)\right\|^2+\left\|\varphi(D_2)\right\|^2}.
\end{equation*}
All results in this paper still hold for this variation. It is worth noting that $\varrho$ is strictly smaller than $\varsigma$ except in case of perfect similarity or dissimilarity. This is proved easily by applying AM-GM inequality \cite{Beckenbach-Bellman} to Cauchy-Schwarz inequality in (\ref{equation: Cauchy-Schwarz}).

Instead of persistence landscapes, the cosine similarity can be defined by employing a different vectorization of persistence diagrams, provided that its images is in an inner product space. There are some vectorizations in the literature that have potential for this purpose, such as Betti curve \cite{Edelsbrunner}, persistence silhouette \cite{Chazal}, life entropy curve \cite{Atienza}, heat kernel \cite{Reininghaus}, persistence image \cite{Adam}, persistence block \cite{Chan}, PD thresholding curve \cite{Chung-Day} and other persistence curves \cite{Chung-Lawson}. However, the perfect similarity and dissimilarity may have different meanings depending on the vectorization.

\backmatter

\bmhead{Acknowledgements}
We are thankful to the referees for constructive comments and suggestions which help to improve the quality of this paper.

\bmhead{Funding}
This work is supported by the research grant DIP-2024-004 provided by Universiti Kebangsaan Malaysia.

\bmhead{Conflict of interest}
The authors declare that they have no conflict of interest.

%
%



\bibliography{sn-bibliography}

\end{document}